\documentclass[11pt]{article}
\usepackage[total={6in, 9in}]{geometry}
\usepackage{mathrsfs,verbatim}
\usepackage[T1]{fontenc}
\usepackage[latin1]{inputenc}
\usepackage[english]{babel}
\usepackage{multicol}
\usepackage{amssymb,amsmath,amsthm,bm}
\usepackage{amsfonts}
\usepackage{latexsym}
\usepackage{bbm}
\usepackage{dsfont}
\usepackage{lineno}
\usepackage{enumitem}
\usepackage{graphicx}
\usepackage[natural]{xcolor}
\usepackage{tikz}
\usepackage{float}
\usepackage{tikz-3dplot}
\usetikzlibrary{shapes}
\usetikzlibrary{arrows,decorations.pathmorphing,backgrounds,positioning,fit,petri,automata}
\usetikzlibrary {positioning}
\usepackage{cases}
\usepackage{mathtools}

\usepackage{algorithm}
\usepackage{algpseudocode}
\usepackage[
pdfauthor={},
pdftitle={},
pdfstartview=XYZ,
bookmarks=true,
colorlinks=true,
linkcolor=blue,
urlcolor=blue,
citecolor=blue,
linktocpage=true,
hyperindex=true
]{hyperref}

\newtheoremstyle{nonitalic}  
{3pt}                      
{3pt}                      
{\normalfont}              
{}                         
{\bfseries}                
{.}                        
{ }                        
{}                         

\theoremstyle{plain}
\newtheorem{thm}{Theorem}[section]

\newtheorem{lem}[thm]{Lemma}
\newtheorem{prob}[thm]{Problem}
\newtheorem{coro}[thm]{Corollary}
\newtheorem{conj}[thm]{Conjecture}

\numberwithin{equation}{section}

\begin{document}

\title{Conflict-free chromatic index of bipartite graphs}

\author{Yuxin Jin\footnote{School of Mathematics and Statistics, Lanzhou University, Lanzhou 730000, China.
		Email: {\tt 19882505991@163.com}. }
\and
Yuping Gao\footnote{School of Mathematics and Statistics, Lanzhou University, Lanzhou 730000, China.
		Corresponding author.
		Email: {\tt gaoyp@lzu.edu.cn}.
		}
}

	\date{}
\maketitle

\begin{abstract}
An edge coloring of a graph $G$ is called conflict-free if, for every edge, its closed neighborhood contains a color that appears exactly once. The least number of colors required for such a coloring is the conflict-free chromatic index of $G$, denoted by $\chi'_{CF}(G)$. Kamyczura, Meszka, and Przyby{\l}o conjectured that $\chi'_{CF}(G)\le 3$ for any bipartite graph $G$ without isolated vertices. In this paper, we confirm this conjecture.

\end{abstract}

\section{Introduction}\label{sec:introduction}

In this paper, all graphs considered are finite and simple. Let $G$ be a graph. A \emph{vertex} (resp. \emph{edge}) $k$-\emph{coloring} $\varphi$ of $G$ is an assignment of $k$ colors $1,2,\ldots,k$ to the vertices (resp. edges) of $G$. If any two adjacent vertices (resp. edges) receive distinct colors, then we call $\varphi$ \emph{proper}. The least integer $k$ such that $G$ admits a proper vertex $k$-coloring (resp. a proper edge $k$-coloring) is called the \emph{chromatic number} (resp. \emph{chromatic index}) of $G$, denoted by $\chi(G)$ (resp. $\chi'(G)$).

Motivated by frequency assignment in cellular networks, Even, Lotker, Ron, and Smorodinsky~\cite{EvenLotkerRonSmorodinsky2003} and Smorodinsky~\cite{S2003} initiated the study of the conflict-free chromatic number of graphs. Let $G$ be a graph. For a vertex $v\in V(G)$, denote by $N_G(v)$ the set of neighbors of $v$ and let $N_G[v]:=N_G(v)\cup\{v\}$. The subscript $G$ may be omitted if it is clear from the context. The least integer $k$ such that $G$ admits a vertex $k$-coloring in which a unique color appears in $N_G[v]$ for every vertex $v\in V(G)$ is called the \emph{conflict-free chromatic number} of $G$, and is denoted by $\chi_{CF}(G)$.

D\k{e}bski and Przyby{\l}o~\cite{DebskiPrzybylo2022} proposed the edge variant, known as conflict-free chromatic index. Let $G$ be a graph and $\varphi$ be an edge $k$-coloring of $G$. We denote by $E_G(v)$ the set of edges incident with a vertex $v\in V(G)$, and define $E_G[uv] := E_G(u)\cup E_G(v)$ to be the \emph{closed neighborhood} of any given edge $uv\in E(G)$. The subscript $G$ may be omitted if clear from context. An edge $uv$ is \emph{satisfied} by $\varphi$ if there exists a color which appears exactly once in $E_G[uv]$. We make a convention that we use the same terminology when $\varphi$ is a partial edge coloring, that is, when $\varphi$ assigns colors only to a subset of $E(G)$. The coloring $\varphi$ is said to be \emph{conflict-free} if all edges of $G$ are satisfied. The least number of colors required for such a coloring is the \emph{conflict-free chromatic index} of $G$, denoted by $\chi'_{CF}(G)$.

Since every proper edge coloring is conflict-free, we have $\chi'_{CF}(G)\le \chi'(G)$. Kamyczura and Przyby{\l}o~\cite{KP2026} obtained an asymptotically tight upper bound for regular graphs. Kamyczura, Meszka, and Przyby{\l}o~\cite{KamyczuraMeszkaPrzybylo2024} proved that $\chi'_{CF}(G)\le \lceil 3\log_2 \Delta(G)\rceil+1$ for any graph $G$ and $\chi'_{CF}(G)\le 4$ for bipartite graphs. For complete bipartite graphs, the authors observed that $\chi'_{CF}(K_{n,m})=3$ whenever $\min\{n,m\}\ge 3$. Therefore, a universal upper bound of $3$ for bipartite graphs would be best possible. This observation leads to the following conjecture.

\begin{conj}[Kamyczura, Meszka, and Przyby{\l}o~\cite{KamyczuraMeszkaPrzybylo2024}]
\label{conj:bipartite}
If $G$ is a bipartite graph without isolated vertices, then $\chi'_{CF}(G)\le 3$.
\end{conj}

In this paper, we resolve Conjecture~\ref{conj:bipartite} in Theorem~\ref{thm:bipartite} by using a new parameter proposed by Kamyczura, Meszka, and Przyby{\l}o~\cite{KamyczuraMeszkaPrzybylo2024}. Let $\varphi$ be an edge coloring of some subgraph $H$ of a given graph $G$. Similarly as before we say that an edge $uv\in E(G)$ is satisfied by $\varphi$ if there exists a color that appears exactly once in $E_G[uv]\cap E(H)$. Denote by $\chi'_{sCF}(G)$ the least number of colors required for coloring a subgraph of $G$
such that all edges of $G$ are satisfied. Given such a coloring, we can use one additional color  for the uncolored edges to obtain a conflict-free edge coloring of $G$. Thus, the following lemma holds.

\begin{lem}[Kamyczura, Meszka, and Przyby{\l}o~\cite{KamyczuraMeszkaPrzybylo2024}]
\label{lem:basic-relation}
For every graph $G$ without isolated edges, $\chi'_{sCF}(G)\le \chi'_{CF}(G)\le \chi'_{sCF}(G)+1$.
\end{lem}

\begin{thm}\label{thm:bipartite}
If $G$ is a bipartite graph without isolated vertices, then $\chi'_{sCF}(G)\le 2$ and $\chi'_{CF}(G)\le 3$.
\end{thm}

By Theorem~\ref{thm:bipartite}, we can obtain two new upper bounds for general graphs in Theorems~\ref{thm:general-bound} and~\ref{thm:Delta}. The first
bound is tight, as demonstrated by $K_{n,m}(n\ge m\ge 3)$.

\begin{thm}
\label{thm:general-bound}
If $G$ is a graph without isolated vertices, then $\chi'_{sCF}(G)\le 2\lceil\log_2\chi(G)\rceil$ and $\chi'_{CF}(G)\le 2\lceil\log_2\chi(G)\rceil+1$.
\end{thm}

\begin{coro}
\label{thm:Delta}
If $G$ is a graph without isolated vertices, then $\chi'_{CF}(G)\le 2\lceil\log_2\Delta(G)\rceil+1$.
\end{coro}

By Brooks' Theorem, Corollary~\ref{thm:Delta} follows directly from Theorem~\ref{thm:general-bound} for all graphs that are neither complete graphs nor odd cycles. The results for complete graphs and cycles are established in the following two lemmas, so we omit the proof of Corollary~\ref{thm:Delta}. Let $K_n$ and $C_n$ denote the complete graph and the cycle on $n$ vertices, respectively.

\begin{lem}[Kamyczura and Przyby{\l}o~\cite{KP2026}]
\label{lem:complete-delta}
For all $n\ge 2$, $\chi'_{CF}(K_n)\le \lceil\log_2(n-1)\rceil+1$.
\end{lem}

\begin{lem}[Kamyczura, Meszka, and Przyby{\l}o~\cite{KamyczuraMeszkaPrzybylo2024}]
\label{lem:cycle}
For all $n\ge 3$, $\chi'_{CF}(C_n)=2$.
\end{lem}

For a tree $T$, D\k{e}bski and Przyby{\l}o~\cite{DebskiPrzybylo2022}, and independently Kamyczura, Meszka, and Przyby{\l}o~\cite{KamyczuraMeszkaPrzybylo2024}, proved that $\chi'_{CF}(T)\le 3$. The following problem remains open.

\begin{prob}[Kamyczura, Meszka, and Przyby{\l}o~\cite{KamyczuraMeszkaPrzybylo2024}]
\label{prob:tree-three}
Characterize the family of all trees $T$ with $\chi'_{CF}(T)=3$.
\end{prob}

Problem~\ref{prob:tree-three} was partially answered by Guo, Li, Li, and Li~\cite{GuoLiLiLi2025}, who gave a linear-time algorithm for trees without vertices of degree $2$. For general graphs $G$, Li~\cite{L2025} proved that it is NP-complete to determine whether $\chi'_{CF}(G)=2$, even if $G$ is a bipartite graph.
We present a necessary and sufficient condition for a tree $T$ to satisfy $\chi'_{CF}(T)=2$ in Theorem~\ref{thm:tree}.

\begin{thm}\label{thm:tree}
Let $T$ be a tree with $|E(T)|\ge 2$. Then $\chi'_{CF}(T)=2$ if and only if
there exists a nonempty proper subset $F\subseteq E(T)$ such that, for every edge
$uv\in E(T)$, one of the following holds:
\begin{enumerate}[label={\rm(\roman*)}]
\item if $uv\in F$, then
$d_F(u)+d_F(v)=2$ or $(d_T(u)-d_F(u))+(d_T(v)-d_F(v))=1$; \label{condition:i}
\item if $uv\notin F$, then
$d_F(u)+d_F(v)=1$ or $(d_T(u)-d_F(u))+(d_T(v)-d_F(v))=2$.\label{condition:ii}
\end{enumerate}
\end{thm}

The proofs of Theorem~\ref{thm:bipartite} and Theorem~\ref{thm:general-bound} are given in Section~\ref{sec:proofbi}. In Section~\ref{sec:tree}, we prove Theorem~\ref{thm:tree}.

\section{Upper bounds on conflict-free chromatic index}\label{sec:proofbi}

In this section, we prove our main results using the concept of a $Y$-dominating set for a bipartite graph $G:=G(X,Y)$, introduced in Subsection~\ref{subsec:pre}.

\subsection{Preliminaries}\label{subsec:pre}

Let $G$ be a graph. For a vertex subset $S\subseteq V(G)$, let $G[S]$ be the subgraph of $G$ induced by $S$ and $G-S=G[V(G)\setminus S]$.
Let $G:=G(X,Y)$ be a bipartite graph. A subset $D\subseteq X$ is called a \emph{$Y$-dominating set} if $N(D)=Y$, where $N(D)=\cup_{v\in D}N(v)$. Note that such a set exists as $N(X)=Y$. We call $D$ \emph{minimal} if for every $v\in D$, $N(D\setminus \{v\})\neq Y$. For a $Y$-dominating set $D\subseteq X$ and $x\in D$, a vertex $y\in N(x)$ is called \emph{a private neighbor of $x$ with respect to $D$} if $N(y)\cap D=\{x\}$.
The set of private neighbors of $x$ is denoted by $P(x)$.

\begin{lem}
\label{lem:private-neighbor}
Let $G:=G(X,Y)$ be a bipartite graph and let $D\subseteq X$ be a $Y$-dominating set. Then the following three statements hold.
\begin{enumerate}[label={\rm(\roman*)}, start=1]
\item  $D$ is minimal if and only if for all $x \in D$, we have $P(x)\neq \emptyset$.  \label{lem:(i)}
\item  If $D$ is minimal, then the collection $\mathcal{P}=\{P(x)\colon x\in D\}$ consists of non-empty, pairwise disjoint subsets of $Y$. \label{lem:(ii)}

\item If $D$ is minimal and $P=\cup_{x\in D}P(x)$, then there exists a matching $M$ in $G[D\cup P]$ covering $D$.\label{lem:(iii)}
\end{enumerate}
\end{lem}

\begin{proof}
For statement~\ref{lem:(i)}, $D$ is minimal if and only if for all $x\in D$, the set $D\setminus\{x\}$ is not a $Y$-dominating set. Equivalently, there exists a vertex $y \in Y$ such that $N(y) \cap D = \{x\}$, meaning $y\in P(x)$.

For statement~\ref{lem:(ii)}, we have $P(x) \neq \emptyset$ for all $x \in D$ by~\ref{lem:(i)}. We show that $P(x_1)\cap P(x_2)=\emptyset$ for all distinct $x_1,x_2\in D$. Suppose not, let $y \in P(x_1) \cap P(x_2)$. Then $\{x_1,x_2\}\subseteq N(y) \cap D$, contradicting the definition of private neighbor.

Statement~\ref{lem:(iii)} follows immediately from Statement~\ref{lem:(ii)} by selecting exactly one edge $xy$ for each $x\in D$ such that $y\in P(x)$. This collection of edges forms the required matching $M$.
\end{proof}

\subsection{Bipartite graphs}\label{subsec:bi}

In this subsection, we construct a partial edge 2-coloring of $G$ such that all edges in $G$ are satisfied. Theorem~\ref{thm:bipartite} then follows from Lemma~\ref{lem:basic-relation}.

\begin{proof}[Proof of Theorem~\ref{thm:bipartite}]
We may assume that $G := G(X, Y)$ is a connected bipartite graph.
We prove that $\chi'_{sCF}(G)\le 2$ and then $\chi'_{CF}(G)\le 3$ by Lemma~\ref{lem:basic-relation}.
Let $D \subseteq X$ be a minimal $Y$-dominating set of $G$. By Lemma~\ref{lem:private-neighbor}\ref{lem:(iii)}, $G[D\cup P]$ has a matching $M$ that covers $D$. We construct a partial edge $2$-coloring $\varphi$ of $G$ as follows.
For every $xy\in E(M)$, let $\varphi(xy)=1$.
For every vertex $y \in Y \setminus V(M)$, we choose exactly one vertex $x \in N(y)\cap D$ (which exists since $D$ is a  $Y$-dominating set) and let $\varphi(xy)=2$.
All other edges in $G$ remain uncolored.

We claim that all edges of $G$ are satisfied by $\varphi$ and complete the proof of the theorem.
Note that by our construction, every vertex $y \in Y$ is incident to exactly one colored edge. Let $xy\in E(G)$ be an arbitrary edge with $x\in X$ and $y \in Y$. We show that $E_G[xy]$ contains a unique color in the following two cases.

{\bf Case 1: $x \in D$.}

In this case, $x$ is incident to an edge $xp\in E(M)$ with $\varphi(xp)=1$, where $p\in P(x)$. Note that $xp$ is the only edge in $E_G(x)$ colored by 1, because $\varphi$ assigns color 1 only to edges in $M$, and $M$ is a matching. If there exists some edge $x'y\in E(G)$ with $\varphi(x'y)=1$, then $y\in P(x')$ and $N(y) \cap D = \{x'\}$. Since $x\in N(y)\cap D$, we have $x = x'$. It implies that color 1 appears exactly once in $E_G[xy]$.

{\bf Case 2: $x\in X\setminus D$.}

Since each colored edge is incident to a vertex in $D$, $x$ is not incident to any colored edge.
If $y\in V(M)$, then $y\in P(x')$ for some $x' \in D$, meaning $x'y$ is the only edge in $E_G(y)$ colored by 1. So 1 appears exactly once in $E_G[xy]$. Assume that $y\in Y\setminus V(M)$, then color 2 appears exactly once in $E_G(y)$ by the definition of $\varphi$. Again, since $x$ is not incident to any colored edge, color 2 appears exactly once  in $E_G[xy]$.

In both cases, $xy$ is satisfied. Thus, $\chi'_{sCF}(G) \le 2$.
\end{proof}

\subsection{Upper bound for general graphs}\label{subsec:general}

The proof strategy for Theorem~\ref{thm:general-bound} is similar to the approach used by Kamyczura, Meszka, and Przyby{\l}o~\cite{KamyczuraMeszkaPrzybylo2024}. For the sake of completeness, we include the full proof in this subsection.

\begin{proof}[Proof of Theorem~\ref{thm:general-bound}]
Let $G$ be a graph without isolated vertices. We first prove that $\chi'_{sCF}(G)\le 2\lceil\log_2\chi(G)\rceil$ by induction on $\chi(G)$. If $\chi(G)=2$, then the theorem follows from Theorem~\ref{thm:bipartite}. Assume that $\chi(G)\ge 3$ and the theorem holds for all graphs $G'$ with $\chi(G')<\chi(G)$.

Let $t:=\lceil\log_2\chi(G)\rceil$ and $k:=2^{t-1}$. Then $\chi(G)\le 2k$. Let $\varphi$ be a proper vertex $2k$-coloring of $G$ with color classes $V_1,\ldots,V_{2k}$, where $V_i=\emptyset$ for $\chi(G)<i\le 2k$ if it exists. Let $A:= \cup_{i=1}^k V_i$ and $B:= \cup_{i=k+1}^{2k} V_i$. Let $H$ be the bipartite graph consisting of all edges between $A$ and $B$ in $G$, and let $F$ be the subgraph of $G$ consisting of all the remaining edges, where we do not include potential isolated vertices in $V(H)$ or $V(F)$. Since $E(F)\subseteq E(G[A])\cup E(G[B])$, we have $\chi(F)\le k<\chi(G)$. By the induction hypothesis, $\chi'_{sCF}(F)\le 2\lceil\log_2 \chi(F)\rceil\le 2\lceil\log_2 k\rceil=2(t-1)$. Thus, there exists an edge $2(t-1)$-coloring $\varphi_1$ of a subgraph of $F$ such that all edges in $F$ are satisfied. Furthermore, since $H$ is bipartite, Theorem~\ref{thm:bipartite} implies there exists an edge coloring $\varphi_2$ of  a subgraph of $H$ using colors $2t-1,2t$ such that all edges in $H$ are satisfied. Combining $\varphi_1$ and $\varphi_2$, we obtain $\chi'_{sCF}(G)\le 2t=2\lceil \log_2\chi(G)\rceil$. By Lemma~\ref{lem:basic-relation}, $\chi'_{CF}(G)\le 2\lceil \log_2\chi(G)\rceil+1$.
\end{proof}

\section{Proof of Theorem~\ref{thm:tree}}\label{sec:tree}

In this section, we prove Theorem~\ref{thm:tree} by analyzing the property of conflict-free edge coloring.

\begin{proof}[Proof of Theorem~\ref{thm:tree}]
First, assume that $\chi'_{CF}(T)=2$. Let $\varphi$
be a conflict-free edge $2$-coloring of $T$, and let $F:=\{e\in E(T)\colon \varphi(e)=1\}$.
Since both colors are used on $T$, $F$ is a nonempty proper subset of $E(T)$.
Let $uv\in E(T)$ and let $n_i(uv)$ denote the number of edges colored by $i$ in $E_T[uv]$ for $i\in\{1,2\}$.
Since $\varphi$ is conflict-free, either $n_1(uv)=1$ or $n_2(uv)=1$.

If $uv\in F$, then $n_1(uv)=d_F(u)+d_F(v)-1$ and $n_2(uv)=(d_T(u)-d_F(u))+(d_T(v)-d_F(v))$. So either $d_F(u)+d_F(v)=2$ or $(d_T(u)-d_F(u))+(d_T(v)-d_F(v))=1$. Condition~\ref{condition:i} follows.
If $uv\notin F$, then $n_1(uv)=d_F(u)+d_F(v)$ and $n_2(uv)=(d_T(u)-d_F(u))+(d_T(v)-d_F(v))-1$. So either $d_F(u)+d_F(v)=1$ or $(d_T(u)-d_F(u))+(d_T(v)-d_F(v))=2$. Condition~\ref{condition:ii} follows.

Conversely, assume that there exists a nonempty proper subset $F\subseteq E(T)$ satisfying conditions~\ref{condition:i} and~\ref{condition:ii}. Let $\varphi$ be an edge 2-coloring of $T$ such that $\varphi(e)=1$ if $e\in F$ and $\varphi(e)=2$ otherwise.
For any edge $uv\in E(T)$, let $n_i(uv)$ denote the number of edges colored by $i$ in $E_T[uv]$ for $i\in\{1,2\}$.

If $uv\in F$, then either $n_1(uv)=d_F(u)+d_F(v)-1=1$ or $n_2(uv)=(d_T(u)-d_F(u))+(d_T(v)-d_F(v))=1$ by condition~\ref{condition:i}.
If $uv\notin F$, then  either $n_1(uv)=d_F(u)+d_F(v)=1$ or $n_2(uv)=(d_T(u)-d_F(u))+(d_T(v)-d_F(v))-1=1$ by condition~\ref{condition:ii}.
Therefore, all edges of $T$ are satisfied by $\varphi$ and we have $\chi'_{CF}(T)=2$.
\end{proof}

\bibliographystyle{abbrv}
\bibliography{ref}

\end{document}